\documentclass[10pt]{article}
\usepackage{dcolumn}
\usepackage{bm}
\usepackage{verbatim}       

\usepackage{amsthm}
\usepackage{amssymb}
\usepackage{amscd}

\usepackage{cite}
\usepackage{bm}
\usepackage{amstext}
\usepackage{amsmath}
\usepackage{amsfonts}
\usepackage{units}
\usepackage{mathrsfs}

\newcommand{\dd}{{\rm d}}

\newcommand{\bit}{\begin{itemize}}
\newcommand{\eit}{\end{itemize}}
\newtheorem{theorem}{Theorem}[section]
\newtheorem{corollary}[theorem]{Corollary}

\newtheorem{proposition}[theorem]{Proposition}
\theoremstyle{definition}
\newtheorem{definition}[theorem]{Definition}
\theoremstyle{remark}
\newtheorem{remark}[theorem]{Remark}

\begin{document}

\title{Compactification of closed preordered spaces}




\author{E. Minguzzi\thanks{
Dipartimento di Matematica Applicata ``G. Sansone'', Universit\`a
degli Studi di Firenze, Via S. Marta 3,  I-50139 Firenze, Italy.
E-mail: ettore.minguzzi@unifi.it} }

\date{}

\maketitle

\begin{abstract}
\noindent A topological preordered space admits a Hausdorff
$T_2$-preorder compactification  if and only if it is Tychonoff and
the preorder is represented by the family of continuous isotone
functions. We construct the largest Hausdorff $T_2$-preorder
compactification for these spaces and clarify its relation with
Nachbin's compactification. Under local compactness the problem of
the existence and identification of the smallest Hausdorff
$T_2$-preorder compactification is considered.
\end{abstract}




{}
\section{Introduction}

A {\em topological preordered space} is a triple
$(E,\mathscr{T},\le)$ where $(E,\mathscr{T})$ is a topological space
and $\le$ is a {\em preorder} on $E$, namely a reflexive and
transitive relation on $E$. The preorder is an {\em order} if it is
antisymmetric. There are many possible compatibility conditions
between topology and preorder that can be added to this basic
structure. We shall mainly consider the {\em $T_2$-preordered
spaces} ({\em closed preordered spaces}), namely those spaces for
which the graph
\[
G(\le)=\{(x,y): x\le y\},
\]
is  closed in the product topology $\mathscr{T}\times \mathscr{T}$
of $E\times E$. In this work we shall follow Nachbin's terminology
\cite{nachbin65} but we remark that in computer science
$T_2$-ordered spaces are very much studied and called {\em
pospaces}.


A $T_2$-preordered space $E$ is a $T_1$-preordered space in the
sense that for every $x\in E$, $i(x)$ and $d(x)$ are closed where
$i(x)=\{y\in E: x\le y\}$ is the increasing hull and $d(x)=\{y\in E:
y \le x\}$ is the decreasing hull.

We recall that an {\em isotone} function $f:E\to \mathbb{R}$ is a
function such that $x\le y \Rightarrow f(x)\le f(y)$. We shall
mostly work with continuous isotone functions with value in [0,1],
although we could equivalently work with  bounded continuous isotone
functions.

In this work we shall consider the problem of compactification for
$T_2$-preordered spaces. It is understood here that the
compactification $cE$ must be endowed with a preorder $\le_c$ which
induces $\le$ on $E$, namely if $x,y\in E$, then $x\le y$ if and
only if $x \le_c y$. The extended preorder is also demanded to be
closed.

In the ordered case this problem has been solved by Nachbin who
proved  \cite{nachbin65,fletcher82,nada86} that a topological
ordered space admits a $T_2$-order compactification if and only if
it is a {\em completely regularly ordered space}, where a {\em
completely regularly preordered space} is a topological preordered
space for which the following two conditions hold
\begin{itemize}
\item[(i)] $\mathscr{T}$ coincides with the initial topology generated by the set of
continuous isotone functions $f: E\to [0,1]$,
\item[(ii)] $x\le y$ if and only if for every continuous isotone function
$f: E\to [0,1]$, $f(x)\le f(y)$.
\end{itemize}
For future reference let us introduce the equivalence relation
$x\sim y$ on $E$, given by ``$x\le y$ and $y \le x$''. Let
$E/\!\!\sim$ be the quotient space, $\mathscr{T}/\!\!\sim$ the
quotient topology, and let $\lesssim$ be defined by, $[x]\lesssim
[y]$ if $x\le y$ for some representatives (with some abuse of
notation we shall denote with $[x]$ both a subset of $E$ and a point
on $E/\!\!\sim$). The quotient preorder is by construction an order.
The triple $(E/\!\!\sim,\mathscr{T}/\!\!\sim,\lesssim)$ is a
topological ordered space and $\pi: E\to E/\!\!\sim$ is the
continuous quotient projection.

Nachbin  proves \cite[Prop.\ 8]{nachbin65}  that the completely
regularly preordered spaces can be characterized as those
topological preordered spaces $(E,\mathscr{T},\le)$ which come from
a quasi-uniformity $\mathcal{U}$, in the sense that
$\mathscr{T}=\mathscr{T}(\mathcal{U}^{*})$ and $G(\le)=\bigcap
\mathcal{U}$ (see \cite{nachbin65,fletcher82} for details on
quasi-uniformities). Note that for these spaces, by (i) above,
$(E,\mathscr{T})$ is completely regular but not necessarily
Hausdorff (equivalently $T_1$). Nevertheless, from (ii) it follows
that $E$ is a $T_2$-preordered space, hence $T_1$-preordered thus
$[x]=d(x)\cap i(x)$ is closed. We conclude that in a completely
regularly preordered space, $\mathscr{T}$ is $T_1$, and hence
$(E,\mathscr{T})$ is a Tychonoff space, if and only if $\le$ is an
order \cite{nachbin65}.

In this work we look for topological preordered spaces that admit a
Hausdorff $T_2$-preordered compactification. Since the
$T_2$-preorder property is hereditary, and every topological space
that admits a Hausdorff compactification is Tychonoff, the class
that we are considering is contained in the family of
$T_2$-preordered Tychonoff spaces. In fact we shall see  that all
these spaces admit a $T_2$-preorder compactification provided the
family of continuous isotone functions determines the preorder. We
shall then look for the largest Hausdorff $T_2$-preorder
compactification and we shall clarify its connection with Nachbin's
$T_2$-order compactification. We will end the paper with a
discussion of the smallest Hausdorff $T_2$-preorder
compactification.

\section{A motivation: the spacetime boundary}

Since the next sections will be particularly abstract, it will be
convenient to motivate this study mentioning an application. This
author is particularly interested in general relativity, but the
reader will easily find other  applications in closely related
fields, for instance, in  dynamical systems theory.

This author's interest for the compactifications of closed
preordered spaces comes from the well-known problem  of attaching a
boundary to a spacetime (physicists term {\em boundary} what is
known as {\em remainder} in topology). We recall that a spacetime is
a connected, Hausdorff, time oriented Lorentzian manifold and is
denoted $(M,g)$, where $g$ is the Lorentzian metric. In relativity
theory the concept of singularity has proved to be quite elusive.
One would like to attach a boundary to spacetime so as to
distinguish between points at infinity and singularities, where the
distinction is made considering the behavior of the Riemann tensor
near the boundary point (e.g. diverging or not).

There have been numerous attempts to construct such a boundary. We
mention Penrose's conformal boundary \cite{penrose64},  Geroch,
Kronheimer and Penrose's causal boundary \cite{geroch72},  Scott and
Szekeres' abstract boundary \cite{scott94}, and various other
proposals by Budic and Sachs \cite{budic74}, Racz
\cite{racz87b,racz88}, Szabados \cite{szabados88,szabados89}, Harris
\cite{harris98}, Flores \cite{flores06b} and others. Apart for the
case of Penrose's conformal boundary, which cannot be applied in
general, one does not demand that spacetime plus the boundary be
still a manifold. In general, one wishes just to preserve some
notion of continuity and provide a way of extending the causal
relation to the boundary.

The above constructions are often quite involved. I propose a
strategy which takes advantage of the fact that any spacetime is a
topological preordered space. Let us clarify this point. The {\em
causal relation} $J^{+}$ on $M$ is given by the pairs $(x,y)$ of
points of $M$ for which there is a $C^1$ curve $\gamma: [0,1]\to M$,
$\gamma(0)=x$, $\gamma(1)=y$, which is {\em causal}, in the sense
that its tangent vector at any point stays in the future causal cone
of $g$. In general $J^{+}$ might be non-closed, however, there is
another relation, intimately connected with $J^{+}$, which is always
closed: the Seifert's relation $J^{+}_S$ \cite{seifert71,
minguzzi07}. The Seifert relation turns spacetime into a topological
 space endowed with a closed relation and, provided some topological conditions are satisfied, it
is indeed possible to compactify  spacetime along the lines
suggested in this work.

We do not claim that the compactification constructed in this way,
denoted $\beta(E)$, will be the most physical. Indeed, it will add
many more points than intuitively required. Nevertheless, it will
provide an important step since it will dominate any other possible
compactification which, therefore, will be obtainable from
$\beta(E)$ through a suitable identification of the boundary points.
The  possibility of adding a boundary and extending the preorder so
as to  keep its closure  is not known among physicists. It suffices
to say that the boundary constructions mentioned above, either apply
to very special spacetimes, or do not share this property.

We could also try a different approach by first showing that the
spacetime is not only a topological preordered space, but in fact a
quasi-pseudo-metric space, and then completing it with a preorder
generalization of the Cauchy completion. Unfortunately, although we
could prove, using the results of \cite{minguzzi12b}, that most
interesting spacetimes are quasi-pseudo-metrizable, the completion
would depend on the chosen quasi-pseudo-metric. Therefore, this
strategy is not entirely viable unless we prove the existence of
some natural spacetime quasi-pseudo-metric.

Let us  end this section explaining why we have to generalize
Nachbin's compactification to the preordered case, even in those
cases in which $E$ is ordered. A key example is provided by Misner's
spacetime, a 2-dimensional spacetime which retains several features
of the Taub-NUT spacetime \cite{hawking73}. This spacetime has
topology $S^1\times \mathbb{R}$ and metric $g=2\dd \theta \dd t+t\dd
\theta^2$. The line $t=0$ of topology $S^1$ is a closed lightlike
geodesic. Through any point of the region $t\le 0$ passes a closed
causal curve.

The topological space $E$ given by the region $t\ge 0$ of Misner's
spacetime can be endowed with a preorder given by the causal
relation. This relation is closed, and the subset $t>0$ with the
induced topology and preorder is a completely regularly ordered
space (indeed it can be shown to be convex and it is normally
preordered due to the results of \cite{minguzzi11f}). The set $t=0$
represents a natural connected piece which bounds the region $t>0$,
but Nachbin's compactification cannot dominate a compactification
with this piece of boundary since Nachbin's compactification would
be ordered while the set $t=0$ is a closed null geodesic, and hence
any pair of points in this set violates antisymmetry. In summary,
although the region $t>0$ is ordered, its most natural
compactifications are not ordered. Evidently, Nachbin's
compactification is too restrictive for applications, and the order
condition on the compactified space must be relaxed.

\section{Hausdorff $T_2$-preorder compactifications}

Given two topological preordered spaces $(E_1,\mathscr{T}_1,\le_1)$
and $(E_2,\mathscr{T}_2,\le_2)$ the function $H:E_1\to E_2$ is a
{\em preorder homeomorphism} if $H$ is bijective, continuous and
isotone and so is its inverse. We speak of {\em preorder embedding}
if $H$ is a preorder homeomorphism of $E_1$ on its image
$H(E_1)\subset E_2$, where $H(E_1)$ is given the induced topology
and induced preorder.

We are interested in establishing under which conditions a
topological preordered space $(E,\mathscr{T},\le)$ admits a
 preorder compactification, namely a preorder embedding  $c:
E\to cE$ into a compact topological preordered space  $(cE,
\mathscr{T}_{c},\le_c)$ in such a way that $c(E)$ is a dense subset
of $cE$. We shall often identify $E$ with $c(E)$ because $c$ is a
preorder homeomorphism between $E$ and $c(E)$. We shall be
especially interested in Hausdorff $T_2$-preordered
compactifications, that is, in those preorder compactifications for
which $(cE, \mathscr{T}_{c},\le_c)$ is also a Hausdorff
$T_2$-preordered space. Sometimes we shall write that $(cE,
\mathscr{T}_{c},\le_c)$ is a preorder compactification by meaning
with this that the map $c: E\to cE$ is a preorder compactification.

\begin{definition}
If $c_1E$, $c_2E$, are two preorder compactifications of $E$ we
write $c_1\le c_2$ if there is a continuous isotone map $C :c_2E\to
c_1E$ such that $C\circ c_2=c_1$ ($c_1\le c_2$ reads ``$c_2$ {\em
dominates} over $c_1$''). The map $C$ is just an extension to $c_2E$
of the preorder homeomorphism $c_1\circ c_2^{-1}:c_2(E)\to c_1(E)$.
Two preorder compactifications are {\em equivalent} if $c_1 \le c_2$
and $c_2\le c_1$.
\end{definition}

We remark that two compactifications may be such that $c_1E=c_2E$,
$C=Id$, but correspond to different preorders on $c_1E$. In this
case $c_1\le c_2$ means that, because $Id$ must be isotone,
$G(\le_{c_2}) \subset G(\le_{c_1})$ (in our conventions the set
inclusion is reflexive). Intuitively, to enlarge the
compactification means to enlarge the domain $cE$ or to narrow the
preorder $\le_c$ or both. From the definition it follows that the
relation of domination on the set of all the compactification is a
preorder. The next result establishes that it is actually an order
provided we pass to the quotient made by the classes of
compactifications related by preorder homeomorphisms.

\begin{proposition}
If two Hausdorff preorder compactifications $c_1,c_2,$ are
equivalent, then there is a preorder homeomorphism $H: c_2E\to c_1E$
such that $H\circ c_2=c_1$.
\end{proposition}

\begin{proof}

Since $c_1 \le c_2$ there is  a continuous isotone map $C_{12}:
c_2E\to c_1E$ such that $C_{12}\circ c_2=c_1$ and since $c_2 \le
c_1$ there is a continuous isotone map $C_{21}: c_1E\to c_2E$ such
that $C_{21}\circ c_1=c_2$. Applying $C_{12}$ to the latter equation
and using the former equation we get $C_{12}\circ C_{21}\circ
c_1=C_{12}\circ c_2=c_1$ which implies that $ C_{12}\circ
C_{21}\vert_{c_1(E)}=Id_{c_1E}\vert_{c_1(E)}$. Since $c_1(E)$  is
dense in $c_1E$ and $c_1E$ is a Hausdorff space we have that
 $C_{12}\circ C_{21}=Id_{c_1E}$ (e.g. \cite[Cor.\ 13.14]{willard70}).
Arguing with the roles of $1$ and $2$ exchanged we get $C_{21}\circ
C_{12}=Id_{c_2E}$ thus $C_{12}$ and $C_{21}$ are one the inverse of
the other. But they are both isotone thus $H:=C_{12}$ is a preorder
homeomorphism.
\end{proof}

\begin{proposition}
If $c_1,c_2$ are two Hausdorff preorder compactifications of $E$ and
$c_1\le c_2$ then the continuous isotone map $C :c_2E\to c_1E$ such
that $C\circ c_2=c_1$ satisfies $C(c_2 E)=c_1 E$, $C(c_2(E))=c_1(E)$
and $C(c_2E\backslash c_2(E))=c_1E\backslash c_1(E)$.
\end{proposition}

\begin{proof}
The map $C$ is necessarily onto because $C(c_2E)$ is compact and
hence closed and the image of $C$ includes $C(c_2(E))=c_1(E)$ which
is dense in $c_1 E$. The preorder compactifications are
compactifications so that  the last equation follows from
\cite[Theor.\ 3.5.7]{engelking89}.
\end{proof}

Let $f: E \to [0,1]$ be a continuous function on a topological space
$(E,\mathscr{T})$, we shall denote by $\le_f$ the total preorder
given by ``$x\le_f y$ if $f(x)\le f(y)$''. Its graph will be denoted
with $G_f$.

The next proposition establishes some necessary conditions for the
existence of a Hausdorff $T_2$-preorder compactification.

\begin{proposition} \label{bos}
If $(E,\mathscr{T},\le)$ is a subspace of a Hausdorff
$\,T_2$-preordered compact space, then $E$ is a $T_2$-preordered
Tychonoff space and the family of continuous isotone functions
$\mathcal{F}$, $f:E\to [0,1]$, is such that  $x\le y$ if and only if
for every $f\in \mathcal{F}$, $f(x)\le f(y)$ (equivalently
$G(\le)=\bigcap_{f\in \mathcal{F}} G_f$).
\end{proposition}

\begin{proof}
Let $E$ be a subspace of a Hausdorff $T_2$-preordered compact space
which we denote $(E', \mathscr{T}',\le')$. Since every compact
Hausdorff space is Tychonoff and this property is hereditary, we
have that $E$ is Tychonoff. The $T_2$-preorder space property is
also hereditary thus $E$ is $T_2$-preordered. Finally, since every
$T_2$-preordered compact space is normally preordered
\cite{minguzzi11f},  for $x',y'\in E'$, $x'\le y'$ iff $F(x')\le
F(y')$ where $F:E'\to [0,1]$ is any continuous and isotone function
on $E'$ (see e.g.\ \cite[Prop.\ 1.1]{minguzzi11c}). Let
$\mathcal{G}$ be the family of continuous isotone functions, $f:E\to
[0,1]$, which come from the restriction of some continuous isotone
function $F:E'\to [0,1]$. Evidently, for $x,y\in E$, $x\le y$ iff
for every $f\in \mathcal{G}$, $f(x)\le f(y)$. Since $\mathcal{F}$
includes $\mathcal{G}$ and is made of isotone functions the last
claim follows.
\end{proof}

\subsection{The largest Hausdorff $T_2$-preorder compactification}

The next result establishes that the previous necessary conditions
are actually sufficient and that there is  a Hausdorff
$T_2$-preordered compactification which dominates over all the other
Hausdorff $T_2$-preordered compactifications. The locally compact
$\sigma$-compact Hausdorff $T_2$-preordered spaces satisfy these
necessary and sufficient conditions \cite{minguzzi11f}.

\begin{theorem} \label{mbr}
Let $(E,\mathscr{T},\le)$ be a $T_2$-preordered Tychonoff space, let
$\mathcal{F}$ be the family of continuous isotone functions $f:E\to
[0,1]$, and assume that the preorder is represented by the
continuous isotone functions i.e. $G(\le)=\bigcap_{f\in \mathcal{F}}
G_f$. Let $\beta: E\to \beta E$ be the Stone-\v Cech
compactification and let $\tilde{\mathcal{F}}$ be the set of
continuous functions over $\beta E$ obtained from the (unique)
extension\footnote{Note that the extension $\tilde{F}$ is really the
extension of $f\circ \beta^{-1}$.} of the elements of $\mathcal{F}$.
There is a largest Hausdorff $T_2$-preordered compactification of
$(E,\mathscr{T},\le)$ given by $(\beta E, \mathscr{T}_\beta,
\le_\beta)$ where $G(\le_\beta)=\bigcap_{\tilde{f}\in
\tilde{\mathcal{F}}} G_{\tilde{f}}$. Every continuous isotone
function on $E$ extends to a continuous isotone function on $\beta
E$.
\end{theorem}

\begin{proof}

Each graph $G_{\tilde{f}}$ is closed because the functions
$\tilde{f}:\beta E \to[0,1]$ are continuous, thus $G(\le_\beta)$
being the intersection of closed sets is closed. Further the graphs
$G_{\tilde{f}}$ contain the diagonal of $\beta E$, thus
$G(\le_\beta)$ contains the diagonal. Moreover, $\le_{\tilde{f}}$ is
transitive which implies that $\le_\beta$ is transitive and hence a
closed preorder on $\beta E$. For every $f\in \mathcal{F}$, if
$x,y\in E$ then $f(x)\le f(y)$ iff $\tilde{f}(x)\le \tilde{f}(y)$
thus $G(\le)=G(\le_\beta) \cap (E\times E)$ which proves that
$(\beta E, \mathscr{T}_\beta, \le_\beta)$ is a preorder
compactification.

If $f: E\to[0,1]$ is a continuous isotone function on $E$ then its
continuous extension to $\beta E$, $\tilde{f}$, is such that
$\tilde{f}\in \tilde{F}$ and by definition of $\le_\beta$,
$G(\le_\beta)\subset G_{\tilde{f}}$ which means that $\tilde{f}$ is
isotone.

Let $(cE, \mathscr{T}_c, \le_c)$ be another preorder
compactification then, since $(\beta E, \mathscr{T}_\beta)$ is the
largest Hausdorff compactification \cite[Theor.\ 19.9]{willard70}
there is a continuous map $H: \beta E\to cE$ such that $H\circ \beta
=c$. The relation on $\beta E$,  $R:=(H\times H)^{-1}G(\le_c)$ which
is clearly reflexive and transitive is also closed in $\beta E\times
\beta E$ because $H$ is continuous.

The map $H$ extends into a continuous function on $\beta E$ the
preorder homeomorphism $c\circ \beta^{-1}: \beta(E)\to c(E)$ thus
$R\cap (\beta(E)\times \beta(E))=G(\le_\beta) \cap (\beta(E)\times
\beta(E))$, that is, $(\beta \times \beta)^{-1}R=G(\le)$. If a
function $g: \beta E\to [0,1]$ is continuous and  $R$-isotone then
$g\circ \beta: E\to [0,1]$ is continuous and isotone which means
that $g\in \tilde{\mathcal{F}}$ (the extension of a continuous
function to a continuous function on $\beta E$ is unique because
$\beta(E)$ is dense in $\beta E$), that is $g$ is also
$G_\beta$-isotone.

Since $(\beta E, \mathscr{T}_\beta, R)$ is a compact
$T_2$-preordered space it is normally preordered \cite[Theor.\
2.4]{minguzzi11f} thus $R=\bigcap_{g\in \mathcal{G}} G_g$ where the
intersection is with respect to the family $\mathcal{G}$ of all the
continuous $R$-isotone functions on $\beta E$. As we have just
proved, this family is a subset of $\tilde{\mathcal{F}}$ thus
$G(\le_\beta)\subset R$. Since $G(\le_\beta) \subset (H\times
H)^{-1}G(\le_c)$ we conclude that $H$ is isotone and hence that
$c\le \beta$.
\end{proof}

\begin{theorem} \label{gha}
A  Hausdorff $T_2$-preorder compactification $(c E,
\mathscr{T}_c,\le_c)$ which shares the properties
\begin{itemize}
\item[(a)] every continuous function $f:E\to [0,1]$ can be extended to a
continuous function on $cE$,
\item[(b)] every continuous isotone function $f:E\to [0,1]$  can be extended to a
continuous isotone function on $cE$,
\end{itemize}
is necessarily equivalent to $(\beta E,
\mathscr{T}_\beta,\le_\beta)$.
\end{theorem}

\begin{proof}
We already know that $c\le \beta$ because $\beta E$ is the largest
Hausdorff $T_2$-preorder compactification. Since the
compactification $(c E, \mathscr{T}_c)$ shares property (a) it is
equivalent with the Stone-\v Cech compactification $(\beta
E,\mathscr{T}_\beta)$, in particular there is a continuous map $D: c
E \to \beta E$ such that $D\circ c=\beta$. The relation on $c E$,
$R:=(D\times D)^{-1}G(\le_{\beta})$ which is clearly reflexive and
transitive is also closed in $c E\times c E$ because $D$ is
continuous.

$D$ extends into a continuous function on $c E$ the preorder
homeomorphism $\beta\circ c^{-1}: c(E)\to \beta(E)$ thus $R\cap
(c(E)\times c(E))=G(\le_c) \cap (c(E)\times c(E))$, that is, $(c
\times c)^{-1}R=G(\le)$. If a function $g: c E\to [0,1]$ is
continuous and  $R$-isotone then $g\circ c: E\to [0,1]$ is
continuous and isotone which means by property (b) that $g$ is also
$G_c$-isotone (the extension of a continuous function to a
continuous function on $c E$ is unique because $c(E)$ is dense in $c
E$).

Since $(c E, \mathscr{T}_c, R)$ is a compact $T_2$-preordered space
it is normally preordered \cite[Theor.\ 2.4]{minguzzi11f} thus
$R=\bigcap_{g\in \mathcal{G}} G_g$ where the intersection is with
respect to the family $\mathcal{G}$ of all the continuous
$R$-isotone functions on $cE$. As we have just proved, this family
is contained in the family of continuous $G_c$-isotone functions
$\mathcal{C}$, $\bigcap_{g\in\mathcal{C}} G_g \subset R$. Finally,
note that $(c E, \mathscr{T}_c,\le_c)$ is also a compact
$T_2$-preordered space hence  normally preordered and hence with a
preorder represented by the continuous $G_c$-isotone functions,
$G(\le_c)= \bigcap_{g\in\mathcal{C}} G_g $, which implies
$G(\le_c)\subset R$. The inclusion $G(\le_c) \subset (D\times
D)^{-1}G(\le_\beta)$ proves that $D$ is isotone  and hence that
$\beta\le c$.
\end{proof}

Adapting the terminology of Fletcher and Lindgren \cite{fletcher82}
for ordered compactifications we can say that the next result proves
that $(\beta E,\mathscr{T}_\beta,\le_\beta)$ is a {\em strict}
preorder compactification.

\begin{theorem} \label{bnw}
On $(\beta E,\mathscr{T}_\beta)$ the closed preorder $\le_\beta$ is
the smallest closed preorder inducing $\le$ on $E$.
\end{theorem}

\begin{proof}
Let $\le_R$ be another closed preorder such that $R\cap (E\times
E)=G(\le)$. The map $\beta': E \to \beta E$, $\beta'=\beta$, where
$\beta E$ is regarded as the preordered space $(\beta
E,\mathscr{T}_\beta,R)$ is a preorder compactification. Since
$\beta$ is the largest $\beta'\le \beta$, which means that there is
a continuous isotone function $B: \beta E\to \beta' E$ such that
$B\circ \beta=\beta'$. On $\beta(E)$ the map $B$ coincides with
$\beta'\circ\beta^{-1}=\beta\circ\beta^{-1}=Id$, thus $B$ is the
identity over $\beta E$. The fact that it is isotone means
$G(\le_\beta) \subset R$ which is the thesis.
\end{proof}

\begin{theorem}
If $(E,\mathscr{T},\le)$ is a compact Hausdorff $T_2$-preordered
space, then its Hausdorff $T_2$-preorder compactification $\beta: E
\to \beta E$ constructed in Theorem \ref{mbr} is equivalent with the
identity $Id: E \to E$.

\end{theorem}

\begin{proof}
The map $c:E\to E$ where $c=Id_E$ and $(c E, \mathscr{T}_c,\le_c)=(
E, \mathscr{T},\le)$ is a preorder compactification which satisfies
both conditions (a) and (b) of Theorem \ref{gha}, thus the preorder
compactification $Id$ is equivalent to $\beta$.
\end{proof}

The discrete preorder is that preorder for which the increasing hull
of a point is made only by the point (thus it is actually an order).
The indiscrete preorder is that preorder for which the increasing
hull of a point is the whole space. The indiscrete preorder is
closed while the discrete preorder requires the Hausdorffness of the
space, which we assume.

\begin{corollary}
If $\le$ is the discrete (indiscrete) preorder then $(\beta E,
\mathscr{T}_\beta,\le_\beta)$ is the Stone-\v Cech compactification
endowed with the discrete (resp. indiscrete) preorder.
\end{corollary}

\begin{proof}
The discrete preorder $\le_d$ on $\beta E$ is clearly the smallest
closed preorder inducing the discrete preorder $\le$, thus
$\le_d=\le_\beta$.

For the indiscrete case let $x,y\in \beta E$ and let $O_x$, $O_y$ be
neighborhoods of $x$ and $y$ respectively. Since $\beta(E)$ is dense
 there are points $x',y'\in E$ such that $x'\in\beta(E)\cap
O_x$, $y'\in\beta(E)\cap O_y$, from  $\beta^{-1}(x')\le
\beta^{-1}(y')$ since $\beta$ is isotone we get $x' \le_\beta y'$
and since $\le_\beta$ is closed we conclude $x \le y$.
\end{proof}

\subsection{The relation with Nachbin's $T_2$-order
compactification}

In this section we wish to study the relation between the
compactification $\beta: E\to \beta E$ and the Nachbin's
compactification $n: E\to nE$ in those cases in which  $E$ is a
completely regularly ordered space so that the latter
compactification applies. In this case, although $\le$ is an order,
$\le_\beta$ need not be an order. We want to prove that the
Nachbin's compactification is obtained  by taking the quotient with
respect to $\sim_\beta$.

Let $(E/\!\!\sim,\mathscr{T}/\!\!\sim,\lesssim)$ be the quotient
topological preordered space and let  $\pi: E\to E/\!\!\sim$ be the
continuous quotient projection. Every open (closed) increasing
(decreasing) set on $E$ projects to an open (resp.\ closed)
increasing (resp.\ decreasing) set on $E/\!\!\sim$ and all the
latter sets can be regarded as such projections. As a consequence,
$(E,\mathscr{T},\le)$ is a normally preordered space
($T_1$-preordered space) if and only if
$(E/\!\!\sim,\mathscr{T}/\!\!\sim,\lesssim)$ is a normally ordered
space (resp.\ $T_1$-ordered space). Using this fact it is easy to
prove (see \cite[Cor. 4.3]{minguzzi11f})

\begin{theorem} \label{kjd}
If $(E,\mathscr{T},\le)$ is  a compact $T_2$-preordered space, then
{$(E/\!\!\sim,\mathscr{T}/\!\!\sim,\lesssim)$} is a compact
$T_2$-ordered space.
\end{theorem}

We are ready to establish the connection with the Nachbin
$T_2$-order compactification.

\begin{theorem} \label{msm}
Let $(E,\mathscr{T},\le)$ be a $T_2$-preordered Tychonoff  space
such that $E/\!\!\sim$ is a completely regularly ordered space, then
the preorder $\le$ is represented by the continuous isotone
functions on $E$. Let $\beta: E\to \beta E$ be the Hausdorff
\mbox{$T_2$-preorder} compactification constructed in Theorem
\ref{mbr} and let \mbox{$\Pi: \beta E \to \beta E/\!\!\sim_\beta$}
be the quotient projection on the $T_2$-ordered space $(\beta
E/\!\!\sim_\beta, \mathscr{T}_\beta / \!\! \sim_\beta,
\lesssim_\beta)$, then\footnote{The inverse $\pi^{-1}$ is
multivalued but the composition $\Pi\circ \beta\circ \pi^{-1}$ is a
well defined function.} $\Pi\circ \beta\circ \pi^{-1}: E/\!\!\sim\,
\to \beta E/\!\!\sim_\beta$ is a $T_2$-order compactification
equivalent to the Nachbin $T_2$-order compactification $n:
E/\!\!\sim \,\to n(E/\!\!\sim)$. In particular, up to equivalences,
the following diagram commutes
\[
\begin{CD}
E @>\beta>> \beta E \\
@V{\pi}VV @VV{\Pi}V \\
E/\!\!\sim @>n >> n(E/\!\!\sim)
\end{CD}
\]

\end{theorem}

\begin{proof}

The order $\lesssim$ on $E/\!\!\sim$ is represented by the
continuous isotone functions because $E/\!\!\sim$ is completely
regularly ordered. Since for $x,y\in E$, $x\le y$ iff $\pi(x)
\lesssim \pi(y)$, and the continuous isotone functions on $E$ pass
to the quotient, the continuous isotone functions on $E$ represent
$\le$.

The fact that  $(\beta E/\!\!\sim_\beta, \mathscr{T}_\beta / \!\!
\sim_\beta, \lesssim_\beta)$ is $T_2$-ordered follows from Theorem
\ref{kjd}.

The expression $\varphi:= \Pi\circ \beta\circ \pi^{-1}$ gives a well
defined function, indeed suppose $x,y\in E$ project on the same
element $[x]\in E/\!\!\sim$, then $x\sim y$ and since $\beta$ is a
preorder embedding $\beta(x)\sim_\beta \beta(y)$ which implies
$\Pi(\beta(x))=\Pi(\beta(y))$.

The function $\varphi$ is continuous, indeed let $O\subset\beta
E/\!\!\sim_\beta$ be an open subset then $\beta^{-1}(\Pi^{-1}(O))$
is open and if $x\in \beta^{-1}(\Pi^{-1}(O))$ and $y\sim x$ then as
$\beta$ is a preorder embedding $\beta(y)\sim_\beta \beta(x)$,
$\beta(x)\in \Pi^{-1}(O)$ which implies $\beta(y)\in  \Pi^{-1}(O)$
and hence $y\in \beta^{-1}(\Pi^{-1}(O))$. The open set
$\beta^{-1}(\Pi^{-1}(O))\subset E$, being projectable has an open
projection by definition of quotient topology which implies that
$\varphi^{-1}(O)$ is open.

Let us prove that $\varphi$ is isotone. Let $[x]\lesssim [y]$,
$x,y\in E$, then $x\le y$ and, since $\beta$ is a preorder
embedding, $\beta(x)\le_\beta \beta(y)$, and finally
$\Pi(\beta(x))\le_\beta \Pi(\beta(y))$ by definition of quotient
order.

Let us prove that $\varphi$ is injective. Let $[x], [y]\in
E/\!\!\sim$ be such that $\varphi([x])=\varphi([y])$, that is,
$\Pi(\beta(x))= \Pi(\beta(y))$. This equality implies $
\beta(x)\sim_\beta \beta(y)$, and since $\beta$ is a preorder
embedding $x\sim y$, that is, $[x]=[y]$.

Let us prove that $\varphi^{-1}\vert_{\varphi(E/\!\!\sim)}:
\varphi(E/\!\!\sim) \to  E/\!\!\sim$ is isotone. Let $x,y\in E$ and
$\Pi(\beta(x))\lesssim_\beta \Pi(\beta(y))$ then $\beta(x)\le_\beta
\beta(y)$ and, since $\beta$ is a preorder embedding, $x\le y$ which
implies $[x]\lesssim [y]$.

Let us prove that $\varphi$ is an embedding. Since $\pi$ is
continuous, given an open set $N\subset E/\!\!\sim$ we have that
$\pi^{-1}(N)$ is open, thus we have only to prove that $\Pi\circ
\beta$ sends open sets on $E$ of the form $\pi^{-1}(N)$  to open
sets on $\Pi\circ \beta(E)$ with the topology induced from $\beta
E/\!\!\sim_\beta$. Let $O\subset E$ be an open set of the form
$O=\pi^{-1}(N)$ with $N$ open set on $E/\!\!\sim$ and let $x\in O$
(thus $[x]\in N$). Since $E/ \!\!\sim$ is completely regularly
ordered space there are \cite{nachbin65} a continuous isotone
function $\hat{f}:E/ \!\!\sim\,\to [0,1]$ and a continuous
anti-isotone function $\hat{g}:E/ \!\!\sim\, \to[0,1]$ such that
$\hat{f}([x])=\hat{g}([x])=1$ and
$\min(\hat{f}([y]),\hat{g}([y]))=0$ for $[y]\in E\backslash N$.

Let us define $f=\hat{f}\circ \pi$, $g=\hat{g}\circ \pi$, so that
they are respectively continuous isotone and continuous anti-isotone
and such that  $f(x)=g(x)=1$ and $\min(f(y),g(y))=0$ for $y\in
E\backslash O$.

The functions $ f,g(\circ \beta^{-1})$ extend to functions
$\tilde{f},\tilde{g}:\beta E\to [0,1]$ respectively isotone and
anti-isotone (extend $-g$ in place of $g$ and take minus the
extended function).  Since they are isotone or anti-isotone there
are continuous functions $F,G: \beta E/\!\!\sim_\beta \to [0,1]$,
respectively isotone and anti-isotone, such that
$\tilde{f}=F\circ\Pi$, $\tilde{g}=G\circ \Pi$ (continuity follows
from the universality property of the quotient map \cite[Theor.\
9.4]{willard70}).

The function $h=\min(\tilde{f},\tilde{g})=\min(F,G)\circ \Pi$ is
continuous and vanishes on $\beta(E\backslash O)$ and hence
$\min(F,G)$ vanishes on $(\Pi\circ \beta) (E\backslash
O)=\varphi((E/\!\!\sim)\backslash N)$ and equals 1 on
$[\beta(x)]_{\beta}=\varphi(x)$. Since $\varphi$ is injective the
open set $Q=\{[w]_\beta \in \beta E/\!\!\sim_\beta:
\min(F([w]_\beta),G([w]_\beta))>0\}$ contains $\varphi(x)$ and is
such that $Q\cap \varphi(E/\!\!\sim) \subset \varphi(N)$ which
proves, due to the arbitrariness of $[x]$, that $\varphi(N)$ is open
in the topology induced on $\varphi(E/\!\!\sim)$ by $\beta
E/\!\!\sim_\beta$. We infer that $\varphi$ is an embedding and since
it is isotone with its inverse it is a preorder embedding.


If $[z]_\beta \in  (\beta
E/\!\!\sim_\beta)\backslash\varphi(E/\!\!\sim)$ and $W$ is an open
set containing $[z]_\beta$ then $\Pi^{-1}(W)$ is open and since
$\beta$ is a dense embedding there is some $r\in E$ such that
$\beta(r)\in \Pi^{-1}(W)$, thus $[r]\in E/\!\!\sim$ is such that
$\varphi([r])\in W$, that is, $\varphi(E/\!\!\sim)$ is dense in
$\beta E/\!\!\sim_\beta$ and hence $\varphi$ is a $T_2$-order
compactification.

Now, let  $\hat{f}:E/\!\!\sim \to [0,1]$ be a continuous isotone
function, and let $f=\hat{f}\circ \pi$. The function  $f:E\to [0,1]$
is a continuous isotone function and we know that there is a
continuous isotone function $\tilde{f}:\beta E\to [0,1]$ which
extends $f\circ\beta^{-1}:\beta(E)\to [0,1]$. Since $\tilde{f}$ is
isotone there is some continuous isotone function $F: \beta
E/\!\!\sim_\beta \to [0,1]$ (continuity follows from the
universality property of the quotient map) such that
$\tilde{f}=F\circ \Pi$, thus $F$ extends $\hat{f}\circ \varphi^{-1}:
\varphi(E/\!\!\sim) \,\to [0,1]$. Since the Nachbin $T_2$-order
compactification is characterized by this extension property
\cite{nachbin65,fletcher82} it follows that $\varphi$ is equivalent
to $n$.

Finally, $\varphi\circ \pi=( \Pi\circ \beta\circ \pi^{-1})\circ
\pi=\Pi\circ \beta$ which proves that, up to equivalences, the
diagram commutes.

\end{proof}

\begin{corollary}
Let $E$ be a completely regularly ordered space, let $\beta: E\to
\beta E$ be the Hausdorff $T_2$-preorder compactification
constructed in Theorem \ref{mbr} and let $\Pi: \beta E \to \beta
E/\!\!\sim_\beta$ be the quotient projection on the $T_2$-ordered
space $(\beta E/\!\!\sim_\beta, \mathscr{T}_\beta / \!\! \sim_\beta,
\lesssim_\beta)$, then $\Pi\circ \beta: E\to \beta E/\!\!\sim_\beta$
is a $T_2$-order compactification equivalent to the Nachbin
$T_2$-order compactification $n: E\to nE$.
\end{corollary}

\begin{proof}
It follows from  the previous theorem noting that a completely
regularly ordered space is a  $T_2$-preordered Tychonoff space.
\end{proof}

If $E$ is a completely regularly ordered space the preorder
compactification $\beta$ need not be equivalent with the Nachbin
compactification. Consider for instance the interval $[0,1)$ with
the usual topology and order. The Nachbin compactification is given
by $[0,1]$ but $\beta([0,1))$ includes many more points.

\subsection{The smallest Hausdorff $T_2$-preorder compactification}

In this section we make some progress in the problem of finding the
smallest Hausdorff $T_2$-preorder compactification of a topological
preordered space in those cases in which it exists. The problem of
identifying and characterizing the smallest $T_2$-order
compactification was considered in
\cite{mccartan68,mccallion72,kunzi92b,richmond93,liu97b}.

In this section $(E,\mathscr{T},\le)$ is a locally compact
$T_2$-preordered Tychonoff space and $\mathcal{F}$ is the family of
continuous isotone functions $f:E\to [0,1]$. Accordingly with the
necessary conditions singled out in Prop.\ \ref{bos}, we shall
assume that the preorder is represented by the continuous isotone
functions i.e. $G(\le)=\bigcap_{f\in \mathcal{F}} G_f$.

Let $\mathcal{C}$, $\mathcal{C}^-$ and $\mathcal{C}^+$ be the
families of continuous functions  in $[0,1]$ which are  constant
outside a compact set, which have compact support and which have
value 1 outside a compact set, respectively.

For every $\mathcal{H}\subset \mathcal{F}$ such that
$G(\le)=\bigcap_{h\in \mathcal{H}}G_h$ we can construct a
$T_2$-preorder compactification $(cE,\mathscr{T}_c,\le_c)$, which we
call $\mathcal{H}$-compactification, through the embedding $c: E\to
[0,1]^{\mathcal{H}\cup \mathcal{C}}$ identifying $cE$ with the
closure of the image. Indeed,  the family $\mathcal{H}\cup
\mathcal{C}$ separates points and has an initial topology coincident
with $\mathscr{T}$ (thanks to local compactness and the inclusion of
$\mathcal{C}$ in the family) thus $c$ is an embedding \cite[Theor.\
8.12]{willard70}. The topology $\mathscr{T}_c$ is that induced from
the product topology in $[0,1]^{\mathcal{H}\cup \mathcal{C}}$ on
$cE$.

We define the $T_2$-preorder $\preceq$ on $[0,1]^{\mathcal{H}\cup
\mathcal{C}}$ as that given by $x\preceq y$ iff $x_h\le_h y_h$ for
every $h\in \mathcal{H}$, where $\le_h$ is the usual order on the
h-th factor [0,1]. This preorder is  closed because the projections
$\pi_h : [0,1]^{\mathcal{H}\cup \mathcal{C}} \to \mathbb{R}$ are
continuous, and hence $G(\preceq)=\bigcap_{h\in
\mathcal{H}}(\pi_h\times \pi_h)^{-1}G(\le_h)$ is closed. It is a
preorder rather than an order because two points can have the same
$h$-components while being different. The $T_2$-preorder $\le_c$ on
$cE$ is that induced by $\preceq$ and is again closed because of the
hereditarity of the $T_2$-preorder property. Finally, $c: E\to c(E)$
is isotone with its inverse because $G(\le)=\bigcap_{h\in
\mathcal{H}}G_h$.

Observe that $h\circ c^{-1}: c(E)\to [0,1]$ extends to the
continuous isotone function $\pi_h\vert_{cE}$, that is, the
continuous isotone functions belonging to $\mathcal{H}$ are
extendable to the $\mathcal{H}$-compactification $cE$ keeping the
same properties.

\begin{remark} \label{nux}
The just defined  $\mathcal{H}$-compactification gives back the
usual one-point compactification  if the preorder $\le$ is
indiscrete and  $\mathcal{H}$ is chosen empty (the additional point
is that of coordinates $f_c$, $c\in \mathcal{C}$, where $f_c$ is the
constant value taken by $c$ outside a compact set).

If the preorder $\le$  is discrete and $\mathcal{H}$ is chosen to
coincide with $\mathcal{C}$ then the compactified space is still the
one-point compactification but endowed with the discrete preorder.
If $\mathcal{H}$ is chosen equal to $\mathcal{C}^-$, then the added
point is less than any other point. If $\mathcal{H}$ is chosen equal
to $\mathcal{C}^+$, then the added point is greater than any other
point.
\end{remark}

In the next proofs we shall often identify $c(E)$ with $E$
especially when referring to the extension of functions.

\begin{proposition} \label{dhd}
Let $c:E\to cE$ be a $\mathcal{H}$-compactification. The remainder
$cE\backslash c(E)$ endowed with the preorder induced from $\le_c$
is a $T_2$-ordered space.
\end{proposition}

\begin{proof}
Since the $T_2$-preorder property is hereditary the remainder is a
$T_2$-preordered  space. Let $x,y\in cE\backslash c(E)$ and suppose
that $x\le_c y\le_c x$ then $x\preceq y \preceq x$, that is for the
(necessarily unique as $c(E)$ is dense in $cE$) continuous isotone
extension $H: cE\to [0,1]$, $H=\pi_h\vert_{cE}$, of
$h\in\mathcal{H}$ we have $H(x)\le H(y)\le H(x)$, which reads
$H(x)=H(y)$. We have only to prove that for every $f\in
\mathcal{C}$, $\pi_f(x)=\pi_f(y)$ from which it follows $x=y$. But
by local compactness $c(E)$ is open in $cE$ thus $cE\backslash c(E)$
is compact and can be separated by open sets (as $cE$ is Hausdorff
and compact hence normal) from the compact set outside which $f$ is
constant. Thus the extension $\pi_f\vert_{cE}$ of $f\in \mathcal{C}$
takes a constant value  on the whole remainder, which implies
$\pi_f(x)=\pi_f(y)$.
\end{proof}

We remark that the previous result does not imply that if $\le$ is
an order then $\le_c$ is an order, but only that if $x\le_c y\le_c x
$, then one point among $x$ and $y$ belongs to $c(E)$ while the
other belongs to $cE\backslash c(E)$.

\begin{proposition} \label{mqo}
Let  $(E,\mathscr{T},\le)$ be a locally compact $T_2$-preordered
Tychonoff space then every $T_2$-preordered Hausdorff
compactification $c:E\to cE$ dominates a
$\mathcal{H}$-compactification for a family $\mathcal{H}\subset
\mathcal{F}$ where $\mathcal{H}$ is such that $G(\le)=\bigcap_{h\in
\mathcal{H}}G_h$. The family $\mathcal{H}$  is made by those
continuous isotone function with value in [0,1] in $E$ that extend
with the same properties to $cE$.
\end{proposition}

\begin{proof}
Let $c_1:E\to c_1E$ be a $T_2$-preordered Hausdorff
compactification. Since $(c_1E, \mathscr{T}_{c_1},\le_{c_1})$ is a
compact $T_2$-preordered space it is normally preordered, thus the
family of continuous isotone functions with values in $[0,1]$,
$\mathcal{H}_{c_1}$, is such that for $x,y\in c_1E$, $x\le_{c_1} y$
if and only if for every $F\in \mathcal{H}_{c_1}$ we have $F(x)\le
F(y)$. Let $\mathcal{H}$ be made by those functions which are the
restriction of the elements of $\mathcal{H}_{c_1}$ to $E$. With this
definition $G(\le)=\bigcap_{h\in \mathcal{H}}G_h$. Let $c_2:E\to
c_2E\subset [0,1]^{\mathcal{H} \cup \mathcal{C}}$ be the
$\mathcal{H}$-compactification and let us prove that $c_1$ dominates
$c_2$.

A continuous isotone map $C:c_1E\to c_2E$ such that $C\circ c_1=c_2$
can be constructed as follows. By local compactness $c_1(E)$ is open
and $c_1E\backslash c_1(E)$ is closed and compact. We consider the
family $\mathcal{H}_{c_1}\cup \mathcal{C}_{c_1}$ where
$\mathcal{C}_{c_1}$ is the family of continuous functions with value
in $[0,1]$ on $c_1E$ which are constant outside a compact set
disjoint from $c_1E\backslash c_1(E)$. The restriction of the
elements of the family $\mathcal{C}_{c_1}$ to $c_1(E)$ gives back
$\mathcal{C}$. By definition, the map $C$ sends $x \in c_1E$ to the
point of $[0,1]^{\mathcal{H}_{c_1}\cup \mathcal{C}_{c_1}}$ whose $f$
coordinate is the value $f(x)$, $f\in \mathcal{H}_{c_1}\cup
\mathcal{C}_{c_1}$. This map is continuous \cite[Theor.\
8.8]{willard70} and isotone, where we define the preorder on
$[0,1]^{\mathcal{H}_{c_1}\cup \mathcal{C}_{c_1}}$ as that determined
by the family $\mathcal{H}_{c_1}$. Let us prove that its image is
included in $c_2E$. From the definitions we have that if $x\in
c_1(E)$ then $C(x)$ belongs to $c_2(E)$. As $C$ is continuous, and
$c_1(E)$ is dense in $c_1E$, if $x\in c_1E$ its image $C(x)$ belongs
to the closure of $c_2(E)$ namely to $c_2E$.
\end{proof}

\begin{proposition} \label{pek}
If $\mathcal{H}_2\supset \mathcal{H}_1$ then the
$\mathcal{H}_2$-compactification dominates over the
$\mathcal{H}_1$-compactification.
\end{proposition}

\begin{proof}
Indeed, if $c_2: E\to c_2E\subset [0,1]^{\mathcal{H}_2 \cup
\mathcal{C}}$ is the former and $c_1:E\to c_1 E \subset
[0,1]^{\mathcal{H}_1 \cup \mathcal{C}}$ is the latter preorder
compactification, then there is a continuous isotone map $C :c_2E\to
c_1E$ such that $C\circ c_2=c_1$. This map is the restriction to
$c_2E$ of $\Pi: [0,1]^{\mathcal{H}_2 \cup \mathcal{C}}\to
[0,1]^{\mathcal{H}_1 \cup \mathcal{C}}$ where $\Pi$ identifies
points with the same coordinates belonging to the set $\mathcal{H}_1
\cup \mathcal{C}$.
\end{proof}

Once a $\mathcal{H}$-compactification is given it is well possible
that some $f\in \mathcal{F}\backslash \mathcal{H}$ could be
extendable as a continuous isotone function to the whole
compactification. Let $i(\mathcal{H})$ be the subset of
$\mathcal{F}$ of so extendable functions. This set being larger than
$\mathcal{H}$ has again the property that it represents $\le$.

\begin{proposition} \label{inv}
The $\mathcal{H}$-compactification and the
$i(\mathcal{H})$-compactification are equivalent.
\end{proposition}

\begin{proof}
Since $\mathcal{H}\subset i(\mathcal{H})$ the
$i(\mathcal{H})$-compactification dominates over the
$\mathcal{H}$-compacti-fication. For  the converse let $c_2: E\to
c_2E\subset [0,1]^{\mathcal{H} \cup \mathcal{C}}$ be the
$\mathcal{H}$-compactification and let $c_1: E\to c_1E\subset
[0,1]^{i(\mathcal{H}) \cup \mathcal{C}}$ be the
$i(\mathcal{H})$-compactification. A continuous isotone map
$C:c_2E\to c_1E$ such that $C\circ c_2=c_1$  can be constructed as
follows. All the functions of $i(\mathcal{H})\cup\mathcal{C}$ extend
(uniquely because $c_2(E)$ is dense in $c_2E$) from $E$ to $c_2E$
thus to every $x\in c_2E$ we assign the image $C(x)$ given by the
point of $ [0,1]^{i(\mathcal{H}) \cup \mathcal{C}}$ having as
coordinates the values taken by the functions belonging to
$i(\mathcal{H})\cup\mathcal{C}$. By construction $C$ is continuous
\cite[Theor.\ 8.8]{willard70}. Let us prove that the image is
included in  $c_1E$. From the definitions we have that if $x\in
c_2(E)$ then $C(x)$ belongs to $c_1(E)$. As $C$ is continuous, and
$c_2(E)$ is dense in $c_2E$, if $x\in c_2E$ its image $C(x)$ belongs
to the closure of $c_1(E)$ namely to $c_1E$. The fact that $C$ is
isotone follows immediately from the definition of preorder in $
[0,1]^{i(\mathcal{H}) \cup \mathcal{C}}$ and from the fact that the
extension of the function in $i(\mathcal{H})$ to $c_2E$ are, by
assumption, continuous and isotone.
\end{proof}

\begin{corollary} \label{jyh}
Let $P(\mathcal{F})$ denote the family of subsets of $\mathcal{F}$.
The map $i: P(\mathcal{F})\to P(\mathcal{F})$ is idempotent, namely
$i(i(\mathcal{H}))=i(\mathcal{H})$. Furthermore, if
$\mathcal{H}_1\subset \mathcal{H}_2$ then $i(\mathcal{H}_1)\subset
i(\mathcal{H}_2)$.
\end{corollary}

\begin{proof}
If a continuous isotone function $f:E\to [0,1]$ can be extended as a
continuous isotone function to the $i(\mathcal{H})$-compactified
space, i.e. $f\in i(i(\mathcal{H}))$ then, as the
$\mathcal{H}$-compactification and the
$i(\mathcal{H})$-compactification are equivalent, it can be extended
as a continuous isotone function to the $\mathcal{H}$-compactified
space that is $f\in i(\mathcal{H})$.

For the last statement, let $f \in i(\mathcal{H}_1)$ that is $f:E\to
[0,1]$ can be extended as a continuous isotone function $f_1:c_1E\to
[0,1]$ to the $\mathcal{H}_1$-compactified space. But the
$\mathcal{H}_2$-compactification dominates over the
$\mathcal{H}_1$-compactification, that is if $c_2: E\to c_2E$ is the
former and $ c_1: E\to c_1E$ is the latter, there is a continuous
isotone function $C:c_2E\to c_1E$ such that $C\circ c_2=c_1$. The
pullback with $C$ of the extension to $c_1E$, namely $f_2=f_1\circ
C$,  is a continuous isotone extension on $c_2E$ of $f$ thus $f\in
i(\mathcal{H}_2)$.
\end{proof}

\begin{theorem} \label{kls}
The $\mathcal{H}$-compactification   is the smallest Hausdorff
$T_2$-preordered  compactification for which the function belonging
to $\mathcal{H}$ are extendable as continuous isotone functions to
the compactified space.
\end{theorem}

\begin{proof}
Let $c: E\to cE$ be a Hausdorff $T_2$-preordered  compactification
for which the functions belonging to $\mathcal{H}$ are extendable.
By Prop.\ \ref{mqo} the compactification $c$ dominates a
$\mathcal{G}$-compactification where $\mathcal{G}$ is the set of
continuous isotone functions on $E$ with value in [0,1] which are
extendable with these properties to $cE$. Thus $\mathcal{H}\subset
\mathcal{G}$ and by Prop.\ \ref{pek} the
$\mathcal{G}$-compactification dominates over the
$\mathcal{H}$-compactification, thus $c$ dominates the
$\mathcal{H}$-compactification.
\end{proof}

\begin{definition}
The {\em family of invariant sets} $\mathcal{I}$ is the set of
subsets $\mathcal{H}\subset \mathcal{F}$ which satisfy
$G(\le)=\bigcap_{h\in \mathcal{H}}G_h$ and are left invariant by
$i$. The set $\mathcal{I}$ is ordered by inclusion.
\end{definition}

The next theorem serves to define the family of continuous isotone
functions $\mathcal{S}$ which characterizes the smallest
compactification.

\begin{theorem}
If the smallest Hausdorff $T_2$-preorder compactification exists
then it is a $\mathcal{S}$-compactification where
$G(\le)=\bigcap_{h\in \mathcal{S}}G_h$, $i(\mathcal{S})=\mathcal{S}$
and $\mathcal{S}=\bigcap \mathcal{I}$.
\end{theorem}

\begin{proof}
Suppose that there is a Hausdorff $T_2$-preorder compactification
which is dominated by all the other Hausdorff $T_2$-preorder
compactifications, then by Prop.\ \ref{mqo} it is equivalent to a
$\mathcal{S}$-compactification where $\mathcal{S}\subset
\mathcal{F}$ is such that $G(\le)=\bigcap_{h\in \mathcal{S}}G_h$.

By Prop.\ \ref{inv} $\mathcal{S}$ can be chosen such that
$\mathcal{S}=i(\mathcal{S})$, thus belonging to $\mathcal{I}$.
Clearly, $\bigcap \mathcal{I}\subset \mathcal{S}$ because
$\mathcal{S}\in \mathcal{I}$. Suppose that $\mathcal{H}\in
\mathcal{I}$ and that $f\in \mathcal{F}$, $f\notin
\mathcal{H}=i(\mathcal{H})$. This means that $f$ is not extendable
as a continuous isotone function to the $\mathcal{H}$-compactified
space. If $C$ is the continuous isotone map from the
$\mathcal{H}$-compactified space to the $\mathcal{S}$-compactified
space (as the $\mathcal{S}$-compactification is dominated by all the
other compactifications) one has that if $f$ were extendable to the
$\mathcal{S}$-compactified space then by pullbacking the extension
to the $\mathcal{H}$-compactified space through $C$ one would get an
extension in the $\mathcal{H}$-compactified space. The contradiction
proves that $f\notin i(\mathcal{S})=\mathcal{S}$ thus
$\mathcal{S}\subset \mathcal{H}$, and finally $\mathcal{S}\subset
\bigcap \mathcal{I}$.
\end{proof}

\begin{remark} \label{bhx}
The smallest compactification does not necessarily exist.  For
instance, if $E$ is non-compact and endowed with the discrete
preorder, the $\mathcal{C}$-compactification dominates over the
$\mathcal{C}^-$-compactification and the
$\mathcal{C}^+$-compacti\-fi\-ca\-tion (see Remark \ref{nux}),
indeed $\mathcal{C}^\mp\subset \mathcal{C}$ see Prop.\ \ref{pek}.
Stated in another way, the one-point compactification endowed with
the discrete preorder dominates over that in which the added point
is less (resp.\ greater) than any other point (indeed, the former
has a smaller preorder). However, $\mathcal{C}^+$ is not contained
in $i(\mathcal{C}^-)$ and conversely, thus the $\mathcal{C}^--$ and
$\mathcal{C}^+-$ compacti\-fi\-ca\-tions differ. Actually, it is
easy to realize that they are minimal, thus there is no smallest
compactification.
\end{remark}

\section{Conclusions}

We have investigated the compactification of topological preordered
spaces, showing the existence of a largest Hausdorff $T_2$-preorder
compactification for every $T_2$-preordered Tychonoff space for
which the preorder is represented by the continuous isotone
functions. An interesting subclass of this family is that of locally
compact $\sigma$-compact Hausdorff $T_2$-preordered spaces
\cite{minguzzi11f}.
 It turns out that this largest compactification is
essentially the Stone-\v Cech compactification endowed with a
suitable preorder. It can be characterized as the Hausdorff
$T_2$-preorder compactification for which all the continuous
function can be continuously extended and the continuous isotone
function do so preserving the isotone property. If the preorder is
an order or the quotient space is a completely regularly ordered
space it is also possible to show a clean relation with Nachbin's
$T_2$-order compactification.

We have considered the problem of identifying the smallest Hausdorff
$T_2$-preorder compactification whenever it exists. We have shown
that it  corresponds necessarily to the compactification obtained
demanding the extendibility of a suitable set of continuous isotone
functions. Generically, this set $\mathcal{S}$ is expected to be
strictly included in the  full set $\mathcal{F}$ of continuous
isotone functions with value in [0,1].

The  approach followed in this work  relies on the study of
continuous isotone functions and their extension properties. We
close noting that filter approaches are also possible. For instance
Choe and Park \cite{choe79} have constructed a Wallman type preorder
compactification which has been subsequently extensively
investigated in \cite{kent85,kent90,kent93,kent95,kunzi08} together
with some variations. For instance, in \cite{kent95} the authors
show that it is possible to obtain the Nachbin compactification from
the Wallman compactification by identifying the points that take the
same value on continuous isotone functions. We have followed a
similar procedure to show that the Nachbin compactification $nE$ can
be obtained from the same functional quotient starting from $\beta
E$.

\section*{Acknowledgments}
I thank a referee for pointing out Remark \ref{bhx}. This work has
been partially supported by ``Gruppo Nazionale per la Fisica
Matematica''  (GNFM) of ``Instituto Nazionale di Alta Matematica''
(INDAM).



\end{document}